\documentclass[12pt]{amsart}

\usepackage{color}
\usepackage{amssymb}
\usepackage{graphics}
\usepackage{tikz}
\usepackage{multicol}

\usepackage{subfigure}
\usepackage{bm}

\numberwithin{equation}{section} 
\newtheorem{thm}{Theorem}[section]
\newtheorem{lem}[thm]{Lemma}
\newtheorem{prop}[thm]{Proposition}
\newtheorem{cor}[thm]{Corollary}

\usepackage[colorlinks,citecolor=red,pagebackref,hypertexnames=false]{hyperref}

\newcommand{\R}{\mathbb{R}}
\newcommand{\N}{\mathbb{N}}

\theoremstyle{plain}

\usepackage{mathtools,mathrsfs,enumitem}
\allowdisplaybreaks

\newcommand{\diam}{\operatorname{diam}}
\newcommand{\dist}{\operatorname{dist}}
\newcommand{\id}{\operatorname{id}}

\begin{document}
\baselineskip 16pt

\title[Inhomogeneous $C^1$ attractors]{Upper box dimension of inhomogeneous attractors of $C^1$ iterated function systems}

\author{Yu-Feng Wu}
\date{}

\address[]{School of Mathematics and Statistics\\HNP-LAMA\\Central South University\\ Changsha, 410083, P. R. China}
\email{yufengwu.wu@csu.edu.cn}



\begin{abstract}
We prove that the upper box dimension of an inhomogeneous attractor of a $C^1$ iterated function system in $\R^d$ is bounded above by the maximum of the singularity dimension of the system and the upper box dimension of the condensation set. This extends the upper bound of Burrell and Fraser~\cite{BurrellFraser} for inhomogeneous self-affine sets  to the $C^1$ setting. The covering argument of Burrell and Fraser does not
apply directly to nonlinear compositions. We overcome this difficulty by combining their approach with an extension of the recursive covering estimate of Feng and Simon~ \cite{FengSimonI}. We also obtain the corresponding bound for orbital sets driven by arbitrary compact subshifts.  
\end{abstract}

\maketitle

\section{Introduction}

Dimension theory of iterated function systems (IFSs) is a central theme in fractal geometry.
For self-conformal (including self-similar) IFSs, the
dimensions of the attractors are well understood under suitable separation conditions; see \cite{Patzschke1997} for classical results and
\cite{Hochman2014,Rapaport2024}
for more recent developments.  Despite substantial progress, the dimension theory of general non-conformal IFSs remains less complete. The work of Feng and Simon \cite{FengSimonI,FengSimonII}
on $C^1$ attractors is particularly relevant here.
In \cite{FengSimonI}, they proved that the upper box dimension
of a $C^1$ attractor is bounded above by the singularity dimension
of the system, extending a classical result of Falconer~\cite{Falconer1988} on self-affine sets.

In 1985, Barnsley and Demko \cite{BarnsleyDemko} introduced
a natural extension of the IFS construction by adjoining
a fixed compact set, called the condensation set, at each
stage of the iteration.
The resulting set, referred to as an inhomogeneous attractor, consists of the homogeneous attractor,
the condensation set, and its images under all finite
compositions of the defining maps. Its dimension therefore reflects both the underlying IFS
and the geometry of the condensation set.
The purpose of this paper is to obtain an upper bound for
the upper box dimension of inhomogeneous attractors
of $C^1$ IFSs in terms of the singularity dimension of the system
and the upper box dimension of the condensation set.

We now introduce the setting and notation.
Let $Z\subset\R^d$ be a non-empty compact set. We call a finite collection $\Phi=\{f_i\}_{i=1}^{\ell}$ of contracting maps,  with $f_i(Z)\subset Z$ for every $i$, an \emph{iterated function system (IFS)}.  By the classical theorem of Hutchinson \cite{Hutchinson},
there is a unique non-empty compact set $K\subset Z$ satisfying
\[
K=\bigcup_{i=1}^{\ell}f_i(K).
\]
We call $K$ the \emph{homogeneous attractor} of $\Phi$.

Given a compact set $C\subset Z$,
there is a unique non-empty
compact set $K_C\subset Z$ satisfying
\begin{equation}\label{eq:InhomogeneousIdentity}
	K_C=C\cup\bigcup_{i=1}^{\ell}f_i(K_C).
\end{equation}
The set $K_C$ is called the \emph{inhomogeneous attractor}
associated with $\Phi$ and $C$, and $C$ is called the
\emph{condensation set}. The existence and uniqueness of $K_C$ follow from the same contraction argument as in Hutchinson's theorem; see, for example, \cite{Snigireva2008}. 

The inhomogeneous attractor admits the following
orbital representation:
\begin{equation}\label{eqOrbital}
K_C=K\cup C\cup\bigcup_{n=1}^{\infty}\bigcup_{i_1,\ldots,i_n\in\{1,\ldots,\ell\}}f_{i_1}\circ\cdots \circ f_{i_n}(C). 
\end{equation}
This was proved for self-similar
IFSs in \cite[Lemma~3.9]{Snigireva2008}.
As observed in \cite[Section~1.1, equation~(3)]{Fraser2012},
the same argument applies to finite IFSs of Lipschitz
contractions on compact metric spaces.

When the maps in $\Phi$ are similitudes, $K_{C}$ is called an inhomogeneous self-similar set. Inhomogeneous self-affine and self-conformal sets are defined analogously. In this paper, we are interested in the case when $\Phi$ consists of $C^1$ maps. More precisely, we assume that there is a bounded open set $U\supset Z$
such that each $f_i$ extends to a $C^1$ diffeomorphism
$f_i:U\to f_i(U)\subset U$, 
and there is a constant $0<\gamma<1$ such that
\begin{equation}\label{eq:contraction}
	|f_i(x)-f_i(y)|\leq\gamma|x-y|
	\qquad
	(x,y\in U,\ 1\leq i\leq\ell).
\end{equation}  
We call such a system $\Phi$ a \emph{$C^1$ IFS on $Z$}, and $K_C$ an \emph{inhomogeneous $C^1$ attractor}. Throughout the paper, the term $C^1$ IFS on a compact set
is understood in this sense. The extensions to $U$ are fixed throughout.
 
In the study of inhomogeneous attractors, 
a basic question is how the dimension of $K_C$ depends on
the dimensions of $K$ and $C$.
For Hausdorff dimension, the orbital representation
\eqref{eqOrbital}, together with countable stability
and the Lipschitz property of the maps, immediately gives
\[
\dim_{\mathrm H}K_C
=
\max\{\dim_{\mathrm H}K,\dim_{\mathrm H}C\}.
\]
The same identity holds for packing dimension. 
The corresponding question for upper box dimension is more
delicate because it is not countably stable.
In particular, the accumulation of infinitely many images
of $C$ at different scales requires additional covering estimates. For further results on the dimension theory of
inhomogeneous attractors, we refer to the survey
of Fraser \cite{FraserSurvey}.

For inhomogeneous self-similar sets,
Fraser \cite[Theorem~2.1]{Fraser2012} proved that
the upper box dimension is bounded above by the maximum
of the similarity dimension of the IFS and
the upper box dimension of the condensation set.
If the underlying self-similar IFS satisfies the open set
condition, this yields
\begin{equation*}
	\overline{\dim}_{\mathrm B}K_C
	=
	\max\{
	\dim_{\mathrm B}K,
	\overline{\dim}_{\mathrm B}C
	\}.
\end{equation*}
Wang and Huang \cite{WangHuang} studied upper and lower box
dimensions of inhomogeneous self-conformal sets with overlaps.
More recently, Orgov\'anyi and Rutar \cite{OrgovanyiRutar}
used two-scale branching functions to obtain formulas
for the lower box dimension and the Assouad spectrum
of inhomogeneous self-conformal sets under suitable
separation assumptions.

For IFSs consisting of arbitrary bi-Lipschitz contractions,
Burrell \cite{Burrell} obtained a general upper bound
involving the upper Lipschitz dimension. In the self-affine setting, this quantity uses only the largest singular value of each composition and
can be strictly larger than the affinity dimension. Burrell and Fraser
\cite[Theorem~2.1]{BurrellFraser} proved the more precise estimate
\[
\overline{\dim}_{\mathrm B}K_C
\leq
\max\{\dim_{\mathrm{AFF}}\Phi,
\overline{\dim}_{\mathrm B}C
\},
\]
where $\dim_{\mathrm{AFF}}\Phi$ is the affinity dimension
of the self-affine IFS $\Phi$,  introduced by Falconer~\cite{Falconer1988}.
This suggests asking whether an analogous bound holds
for nonlinear $C^1$ systems, with the affinity dimension
replaced by the singularity dimension.

For a $C^1$ IFS $\Phi$ on $Z$, let $\dim_{\mathrm S}\Phi$ denote its singularity dimension;
the precise definition is given in
Section~\ref{sec:preliminaries}.
Our main result gives an affirmative answer to the question above.

\begin{thm}\label{thm:main}
Let $\Phi=\{f_i\}_{i=1}^{\ell}$ be a $C^1$ IFS on a non-empty compact set $Z\subset\R^d$. For a compact set $C\subset Z$, let $K_C$ be
the associated inhomogeneous attractor.  Then
\[\overline{\dim}_{\mathrm B} K_C\leq \max\{\dim_{\mathrm S}\Phi,\overline{\dim}_{\mathrm B}C\}.\]
\end{thm}

No separation condition or bounded distortion assumption
is required in Theorem~\ref{thm:main}.
For self-affine IFSs, the singularity dimension agrees with
the affinity dimension, so Theorem~\ref{thm:main}
extends the upper bound of Burrell and Fraser. Taking $C=\emptyset$ gives $K_C=K$ and
$\overline{\dim}_{\mathrm B}C=0$, and hence Theorem~\ref{thm:main} also recovers the homogeneous upper bound
of Feng and Simon \cite{FengSimonI}.
Further consequences are discussed in
Section~\ref{sec:consequences}.

We also extend the upper bound to orbital sets whose defining
compositions are restricted by a compact subshift, with the singularity
dimension computed on that subshift; see Theorem~\ref{thm:subshift-orbits}.
This extension follows from Theorem~\ref{thm:main} by constructing
ordinary inhomogeneous attractors from the allowed blocks and using
the pressure approximation of Feng and Simon
\cite[Proposition~3.1]{FengSimonI}.
For IFSs consisting of similarities, the corresponding
orbital construction and upper
box-dimension bound were
studied by Dubey and Verma
\cite{DubeyVerma} in their work on inhomogeneous sub-self-similar sets.

We now outline the proof of Theorem~\ref{thm:main}.
The covering argument of Burrell and Fraser does not apply directly to the $C^1$ setting. In the affine case, the image of a ball under a long composition is an ellipsoid whose axes are exactly described by the singular values of the product matrix. 
This permits direct covering estimates for images
under arbitrarily long compositions.
For nonlinear $C^1$ systems, this exact
description is lost, and the $C^1$ assumption alone
does not provide uniform distortion control for long
compositions.
Feng and Simon~\cite{FengSimonI} overcome this obstruction in the homogeneous $C^1$
setting by applying a recursive covering estimate (cf.~\cite[Proposition~4.2]{FengSimonI}), which is based on a local covering estimate for $C^1$ images of balls (cf.~\cite[Lemma 4.1]{FengSimonI}).

To apply Feng and Simon's method to the inhomogeneous problem,
we extend the recursive covering estimate of
\cite[Proposition~4.2]{FengSimonI},
which concerns images of the homogeneous attractor $K$,
to images of arbitrary compact sets $D\subset Z$.
The resulting estimate, proved in Lemma~\ref{coveringlem},
retains the initial covering number of $D$ explicitly,
with constants independent of $D$.
The key observation is that uniform contraction brings
the relevant images of $D$ close to the homogeneous attractor
after a fixed number of iterates.
The local covering argument can then be continued recursively.
This allows us to control both the images of the condensation
set and those of the ambient set within the same estimate.

The remainder of the proof proceeds in two stages. First, we combine the covering estimate with a stopping
decomposition of the condensation orbit to obtain
an upper bound involving the zero of an additive pressure.
Second, we apply this bound to block IFSs and adapt
the argument of Feng and Simon to show that the
corresponding pressure zeros converge from above
to the singularity dimension.

The paper is organized as follows.
In Section~\ref{sec:preliminaries}, we introduce the notation and definitions.
In Section~\ref{sec:covering}, we establish the covering estimate used in the proof of the main theorem.
In Section~\ref{sec:mainproof}, we use this estimate
to prove Theorem~\ref{thm:main}. Section~\ref{sec:subshift-orbits} gives the extension to orbital sets
driven by subshifts.
Finally, in Section~\ref{sec:consequences}, we present consequences of the main result,
including an application to planar nonlinear IFSs.

\section{Preliminaries}\label{sec:preliminaries}

\subsection{Symbolic notation and the coding map}

Let $\mathcal{A}=\{1,\ldots,\ell\}$ be the alphabet associated to the IFS $\Phi=\{f_i\}_{i=1}^{\ell}$. Let $\Sigma=\mathcal{A}^{\N}$ be the shift space with left shift $\sigma$. Set
\[
\mathcal{A}^*=\bigcup_{n=1}^{\infty}\mathcal{A}^n.
\]
Thus $\mathcal{A}^*$ consists of the non-empty finite words over $\mathcal{A}$. Denote the empty word by $\vartheta$ and set $\mathcal{A}^0=\{\vartheta\}$. 
For a non-empty word $I=i_1\cdots i_n\in\mathcal{A}^n$, write
\[
 |I|=n,
 \qquad
 f_I=f_{i_1}\circ\cdots\circ f_{i_n}, 
\]
and let $I^-=i_1\cdots i_{n-1}$, with $I^-=\vartheta$
when $n=1$.  For the empty word, set $|\vartheta|=0$ and $f_{\vartheta}=\id$, where $\id$ denotes the identity map.  For $\omega=(\omega_j)_{j\geq1}\in\Sigma$ and $n\geq 1$, write $\omega|n=\omega_1\cdots\omega_n$ and $[I]=\{\omega\in\Sigma:\omega|n=I\}$. We adopt the convention that $[\vartheta]=\Sigma$. 

The coding map $\Pi:\Sigma\to K$ is the continuous surjection given by
\[
 \Pi(\omega)=\lim_{n\to\infty}f_{\omega|n}(x_0),
\]
where $x_0\in Z$ is arbitrary.  Note that for every $\omega\in\Sigma$ and $n\geq 1$, we have 
\[\Pi(\omega)=f_{\omega|n}(\Pi(\sigma^n\omega)).\]

\subsection{Covering numbers and box dimension}
Throughout, for $x\in\R^d$ and $r>0$, let $B(x,r)$ denote the closed ball
centred at $x$ with radius $r$.
For a non-empty bounded set $E\subset\R^d$ and $r>0$, let $N_r(E)$ be the least number of closed balls of radius $r$ needed to cover $E$. We also set $N_r(\emptyset)=0$ for every $r>0$. The upper box dimension of $E$ is defined by 
\[
\overline{\dim}_{\mathrm B}E=\limsup_{r\to 0^+}\frac{\log N_r(E)}{-\log r}.
\]
The lower box dimension $\underline{\dim}_{\mathrm B}E$ is defined by replacing
$\limsup$ with $\liminf$.
If the two values coincide, the common value is called
the box dimension of $E$ and is denoted by $\dim_{\mathrm B}E$. By convention, we set $\underline{\dim}_{\mathrm B}\emptyset
=\overline{\dim}_{\mathrm B}\emptyset
=\dim_{\mathrm B}\emptyset=0$.

We shall repeatedly use the elementary consequence that, if $t>\overline{\dim}_{\mathrm B}E$, then for every fixed $r_0>0$ there is a constant $M=M(E,t,r_0)$ such that
\begin{equation}\label{eq:BoxCoverPower}
 N_r(E)\leq Mr^{-t}
 \qquad(0<r\leq r_0).
\end{equation}

\subsection{Pressure and singularity dimension}

For a $d\times d$ real matrix $A\in\R^{d\times d}$, write
$\alpha_1(A)\geq\cdots\geq\alpha_d(A)$
for its singular values. Following \cite{Falconer1988}, for $s\geq0$ we define the singular-value function $\phi^s:\R^{d\times d}\to [0,\infty)$ as 
\begin{equation*}
 \phi^s(A)=
 \begin{cases}
  \alpha_1(A)\cdots\alpha_k(A)\alpha_{k+1}(A)^{s-k},
     & 0\leq s<d, \\
  |\det A|^{s/d},&s\geq d,
 \end{cases}
\end{equation*}
where $k=\lfloor s\rfloor$  is the integral part of $s$. Here $\phi^0(A):=1$. 
The singular-value function is submultiplicative:
\begin{equation}\label{eq:phi-submultiplicative}
 \phi^s(AB)\leq\phi^s(A)\phi^s(B).
\end{equation}

For a $C^1$ map $h:U\to\mathbb R^d$ and $z\in U$,
we write $D_zh$ for the derivative of $h$ at $z$.
For $s\geq0$, define the  continuous subadditive potential $\mathcal{G}^s=\{g_n^s\}_{n\geq1}$ on $\Sigma$ by
\begin{equation}\label{eqGs}
 g_n^s(\omega)
 =\log\phi^s\!\left(D_{\Pi(\sigma^n\omega)}f_{\omega|n}\right),\qquad \omega\in \Sigma.
\end{equation}
The chain rule and \eqref{eq:phi-submultiplicative} imply that 
\[
 g_{m+n}^s(\omega)
 \leq g_m^s(\omega)+g_n^s(\sigma^m\omega).
\]
Its subadditive topological pressure is
\begin{equation*}
 P_{\sigma}(\mathcal{G}^s)
 =\lim_{n\to\infty}\frac1n
  \log\sum_{I\in\mathcal{A}^n}
  \sup_{\omega\in[I]}\exp\bigl(g_n^s(\omega)\bigr).
\end{equation*}
The limit exists by subadditivity. The function $s\mapsto P_\sigma(\mathcal G^s)$ is continuous
and strictly decreasing on $[0,\infty)$, with
$P_\sigma(\mathcal G^0)=\log\ell\geq0$ and
$P_\sigma(\mathcal G^s)\to-\infty$ as $s\to\infty$.
Following \cite[Definition~2.5]{FengSimonI},
the \emph{singularity dimension} of $\Phi$, denoted by
$\dim_{\rm S}\Phi$, is the unique non-negative number $s$
such that $P_\sigma(\mathcal G^s)=0$.

For a continuous function $\psi:\Sigma\to\R$, we also use the classical topological pressure
\begin{equation*}
 P_\sigma(\psi)
 =\lim_{n\to\infty}\frac1n
 \log\sum_{I\in\mathcal{A}^n}
 \exp\left(\sup_{\omega\in[I]}S_n\psi(\omega)\right),
\end{equation*}
where $S_n\psi=\sum_{j=0}^{n-1}\psi\circ\sigma^j$. When the underlying shift space $Y$ and shift map $T$
need to be specified, we write $P(Y,T,\cdot)$.

\section{A covering estimate}\label{sec:covering}

In this section, we establish a covering estimate for
$f_{\omega|n}(D)$, $\omega\in\Sigma$, $n\in\N$,  that is uniform over compact sets $D\subset Z$.
We begin with the following local covering lemma of Feng and Simon.
 
\begin{lem}\cite[Lemma~4.1]{FengSimonI}\label{FSlem}
Let $U\subset \R^d$ be an open set and $f:U\to \R^d$ be a non-degenerate $C^1$ map. Let $E\subset U$ be compact and $k\in \{0,1,\ldots,d-1\}$. 
Then there exists \(r_0 > 0\) such that, for any \(y \in E\), \(z \in B(y, r_0)\) and \(0 < r < r_0\), the set \(f(B(z, r))\) can be covered by at most
\[
(4d)^d \cdot \frac{\phi^k(D_y f)}{(\alpha_{k+1}(D_y f))^k}
\]
balls of radius \(\alpha_{k+1}(D_y f)r\).
\end{lem}

For the remainder of this section, fix $k\in\{0,1,\ldots,d-1\}$.  Define functions $G,H:\Sigma\to(0,\infty)$ by
\begin{equation}\label{eq:GH}
 G(\omega):=
 \frac{(4d)^d\phi^k\!\left(D_{\Pi(\sigma\omega)}f_{\omega_1}\right)}
 {\alpha_{k+1}\!\left(D_{\Pi(\sigma\omega)}f_{\omega_1}\right)^k},
 \qquad
 H(\omega):=\alpha_{k+1}\!\left(D_{\Pi(\sigma\omega)}f_{\omega_1}\right).
\end{equation}
Note that $G,H$ are continuous, $G(\omega)\geq (4d)^d\geq 1$, and
\[
0<H_-:=\min_{\omega\in\Sigma}H(\omega)
\leq H(\omega)\leq\gamma<1.
\]
For $n\geq1$, let
\begin{equation*}
 G_n(\omega)=\prod_{j=0}^{n-1}G(\sigma^j\omega),
 \qquad
 H_n(\omega)=\prod_{j=0}^{n-1}H(\sigma^j\omega).
\end{equation*}
We adopt the convention that $G_0=H_0=1$. 

The following lemma is an extension of
the recursive covering estimate of Feng and Simon
\cite[Proposition~4.2]{FengSimonI}.
Their result concerns images of the attractor $K$,
whereas here the set $D$ may be any compact subset of $Z$,
and its covering number $N_r(D)$ is retained explicitly.
The proof adapts their recursive argument, with constants independent
of $D$.

\begin{lem}\label{coveringlem}
There exist constants $A\geq1$ and $\rho>0$, depending only on $\Phi$, $Z$, and $k$, such that for every compact set $D\subset Z$, every $\omega\in\Sigma$, every $n\geq1$, and every $0<r\leq\rho$,
\begin{equation}\label{eq:coverFD}
 N_{A r H_n(\omega)}\bigl(f_{\omega|n}(D)\bigr)
 \leq A G_n(\omega)N_r(D).
\end{equation}
\end{lem}

\begin{proof}
The case $D=\emptyset$ is immediate, so assume
that $D\neq\emptyset$.
Applying Lemma~\ref{FSlem} to the mappings $f_i$ with $E=K$, we see that there exists $r_0>0$ such that, for any $y\in K$, $z\in B(y,r_0)$ and $0<r<r_0$ and $i\in\{1,\ldots,\ell\}$, the set $f_i(B(z,r))$ can be covered by 
\[
(4d)^d \cdot \frac{\phi^k(D_y f_i)}{(\alpha_{k+1}(D_y f_i))^k}
\]
balls of radius \(\alpha_{k+1}(D_y f_i)r\).
By shrinking $r_0$ if necessary,  we may also assume that the $2r_0$-neighbourhood of $K$ is contained in $U$.

Let $\gamma\in(0,1)$ be the common contraction constant of $\Phi=\{f_i\}_{i=1}^{\ell}$ in \eqref{eq:contraction}. 
Then $H(\omega)\leq\gamma$ for every
$\omega\in\Sigma$.
Choose an integer $n_0\geq1$ such that
\[\gamma^{n_0}\max\{1,\diam Z\}<\frac{r_0}{4}.\]
Next, choose $A\geq1$ sufficiently large that
\[
A\geq
2\max_{0\leq j\leq n_0}
\left(\frac{\gamma}{H_-}\right)^j,
\]
and then choose $\rho>0$ so small that
\begin{equation}\label{eq:rho}
	A\rho\leq1,
	\qquad
	2\rho<\dist(Z,U^c).
\end{equation}
The constants $A$ and $\rho$ depend only on $\Phi, Z$ and $k$, and remain fixed throughout the proof.

Fix a compact set $D\subset Z$ and $0<r\leq\rho$.
We prove \eqref{eq:coverFD} for every $\omega\in\Sigma$
by induction on $n$.

First suppose that $1\leq n\leq n_0$. Let $N=N_r(D)$ and let $B_1,\ldots, B_N$ be closed balls of radius $r$ whose union covers $D$. They all meet $D$ by the minimality of $N_r(D)$.  For each $i\in \{1,\ldots,N\}$, pick $z_i\in B_i\cap D$ and let $\widetilde{B}_i:=B(z_i,2r)$. Then $D\subset \bigcup_{i=1}^N\widetilde{B}_i$, and $\widetilde{B}_i$ are all contained in $U$ by \eqref{eq:rho}.
Moreover, by \eqref{eq:contraction},
\[
f_{\omega|n}(B(z_i,2r))
\subset
B\bigl(f_{\omega|n}(z_i),2\gamma^n r\bigr), \qquad i=1,\ldots, N.
\]
Thus the balls on the right-hand side form a cover of $f_{\omega|n}(D)$.  
By the choice of $A$ and the inequality
$H_n(\omega)\geq H_-^n$, we have
\[
2\gamma^n r
\leq ArH_-^n
\leq ArH_n(\omega).
\]
Consequently,
\[
N_{ArH_n(\omega)}\bigl(f_{\omega|n}(D)\bigr)
\leq N_r(D)
\leq AG_n(\omega)N_r(D),
\]
where the last inequality follows from $A\geq1$ and $G_n(\omega)\geq 1$.
Thus \eqref{eq:coverFD} holds for
$1\leq n\leq n_0$.

Now suppose that, for some $n\geq n_0$,
\eqref{eq:coverFD} holds
for every $\omega\in\Sigma$.
Fix $\omega\in\Sigma$ and set
$E=f_{(\sigma\omega)|n}(D)$
and 
$y=\Pi(\sigma\omega)$.
Applying the induction hypothesis to $\sigma\omega$,
we obtain a cover of $E$ by $M\leq AG_n(\sigma\omega)N_r(D)$ balls $B_1,\ldots,B_M$
of radius
$R=ArH_n(\sigma\omega)$.
Discarding any balls which do not meet $E$, we assume that $B_i\cap E\neq\emptyset$ for $1\leq i\leq M$.
By \eqref{eq:rho} and the choice of $n_0$,
\[
R
\leq A\rho\gamma^n
\leq\gamma^{n_0}
<\frac{r_0}{2}.
\]

Note that $y=f_{(\sigma\omega)|n}
\bigl(\Pi(\sigma^{n+1}\omega)\bigr)$. Also, for every $x\in E$, there exists $z\in D$ such that $x=f_{(\sigma\omega)|n}(z)$. 
Since both $z$ and $\Pi(\sigma^{n+1}\omega)$ belong to $Z$,
we have
\[
\begin{split}
	|x-y|
	&=
	\left|
	f_{(\sigma\omega)|n}(z)
	-
	f_{(\sigma\omega)|n}
	\bigl(\Pi(\sigma^{n+1}\omega)\bigr)
	\right|\\
	&\leq\gamma^n\diam Z
	\leq\gamma^{n_0}\diam Z
	<\frac{r_0}{2}.
\end{split}
\]
Since $B_i$ meets $E$, the distance between its center and $y$ is at most
\[
R+\gamma^n\diam Z<r_0. 
\]
In particular, each $B_i$ is contained in $U$.

We may therefore apply Lemma~\ref{FSlem} to $f_{\omega_1}$
on each $B_i$, using the reference point
$y=\Pi(\sigma\omega)$.
It follows that $f_{\omega_1}(B_i)$ can be covered by
at most $G(\omega)$ balls of radius
\[
H(\omega)R
=
ArH(\omega)H_n(\sigma\omega)
=
ArH_{n+1}(\omega).
\]
Since
\[
f_{\omega|n+1}(D)
=
f_{\omega_1}(E)
\subset
\bigcup_{i=1}^{M}f_{\omega_1}(B_i),
\]
we conclude that
\[
\begin{split}
	N_{ArH_{n+1}(\omega)}
	\bigl(f_{\omega|n+1}(D)\bigr)
	&\leq G(\omega)M\\
	&\leq
	A G(\omega)G_n(\sigma\omega)N_r(D)\\
	&=AG_{n+1}(\omega)N_r(D).
\end{split}
\]
This completes the induction.
\end{proof}

\section{Proof of  Theorem~\ref{thm:main}}\label{sec:mainproof}

\subsection{An upper bound via additive pressure}\label{sec:additive-bound}

In this subsection, we derive an upper bound for $\overline{\dim}_{\rm B}K_C$ in terms of $\overline{\dim}_{\rm B}C$ and the zero of an additive pressure function. 

Fix $k\in\{0,1,\ldots,d-1\}$ and let $G,H$ be the functions defined in \eqref{eq:GH}.  Put
\[
 g=\log G,
 \qquad
 h=\log H.
\]
Then $g,h:\Sigma\to \R$ are continuous functions with $g\geq d\log(4d)>0$ and $h<0$. The pressure considered below is the ordinary topological pressure
of the continuous potential $g+th$, whose Birkhoff sums satisfy
\[
S_n(g+th)=\log G_n+t\log H_n.
\] 
By the standard properties of topological pressure~\cite[Theorem 9.7]{Walters82}, the function
$t\mapsto P_\sigma(g+th)$, $t\in [0,\infty)$, 
is continuous and strictly decreasing from a positive value to $-\infty$ as $t$ ranges from $0$ to $+\infty$.  Let $q>0$ be its unique zero:
\begin{equation*}
 P_\sigma(g+qh)=0.
\end{equation*}

Recall that $ H_n(\omega)=\prod_{j=0}^{n-1}H(\sigma^j\omega)$. 
For $n\in\N$ and $I\in\mathcal{A}^n$, define
\[
\widetilde H(I)=\sup_{\omega\in[I]}H_n(\omega).\]
For the empty word, 
set $\widetilde H(\vartheta)=1$. 
For $0<\delta<1$, let
\[\mathcal{S}_\delta=\bigl\{I\in\mathcal{A}^*:\widetilde H(I)\geq\delta\bigr\},\qquad 
\mathcal{T}_\delta=\bigl\{I\in\mathcal{A}^*:\widetilde H(I)<\delta\leq\widetilde H(I^-)\bigr\}.\]
The following elementary decomposition separates the condensation
orbit pieces before and after their first crossing of the scale
$\delta$.

\begin{lem}\label{lem:decomposition}
	For every $0<\delta<1$, the families $\mathcal{S}_\delta$ and
	$\mathcal T_\delta$ are finite, and the cylinders
	$[I]$, $I\in\mathcal T_\delta$, form a partition of $\Sigma$.
	Moreover,
	\begin{equation}\label{eqKCcoverST}
		K_C
		\subset
		C\cup\bigcup_{I\in\mathcal{S}_\delta}f_I(C)
		\cup\bigcup_{I\in\mathcal T_\delta}f_I(Z).
	\end{equation}
\end{lem}

\begin{proof}
	Fix $0<\delta<1$. 
	Since $H\leq \gamma<1$, for every word $I\in \mathcal{A}^*$ and every
	$i\in\mathcal{A}$, we have $\widetilde H(Ii)\leq \gamma\widetilde H(I)$ and $\widetilde H(I)\leq \gamma^{|I|}$. 
	Choose $N\geq1$ such that $\gamma^N<\delta$.
Then every $\omega\in\Sigma$ has a unique prefix
	$I\in\mathcal T_\delta$, of length at most $N$.
	Also, every word in $\mathcal{S}_\delta$ has length less than $N$.
	Therefore, both $\mathcal{S}_{\delta}$ and $\mathcal{T}_{\delta}$ are finite, and the cylinders $[I]$, $I\in\mathcal{T}_{\delta}$, form a partition of $\Sigma$.
	
	By the orbital representation of $K_C$ (cf. \eqref{eqOrbital}), it remains to show that the right-hand side of \eqref{eqKCcoverST} covers $K$ and
	the sets $f_J(C)$, $J\in\mathcal{A}^*$.
	If $\widetilde H(J)\geq\delta$, then $J\in\mathcal{S}_\delta$.
	Otherwise, write $J=J_1J_2$, where $J_1$ is the unique prefix of $J$ in $\mathcal T_\delta$ and $J_2$ may be empty.
	Since $f_{J_2}(C)\subset Z$,
	\[
	f_J(C)=f_{J_1}(f_{J_2}(C))\subset f_{J_1}(Z).
	\]
	Finally, for $x\in K$, choose $\omega\in\Sigma$
	with $\Pi(\omega)=x$, and let $I$
	be its prefix in $\mathcal T_\delta$.
	Then
	\[
	x=f_I\bigl(\Pi(\sigma^{|I|}\omega)\bigr)
	\in f_I(K)\subset f_I(Z).
	\]
	This proves \eqref{eqKCcoverST}.
\end{proof}

We now establish the additive-pressure upper bound.

\begin{prop}\label{prop:additive-upper}
For every compact condensation set $C\subset Z$,
\begin{equation*}
 \overline{\dim}_{\rm B}K_C\leq\max\bigl\{q,\overline{\dim}_{\rm B}C\bigr\}.
\end{equation*}
\end{prop}

\begin{proof}
Write
$c=\overline{\dim}_{\mathrm B}C$
and fix
$t>\max\{q,c\}$.
Then
$P_\sigma(g+th)<0$.
Consequently, there exist constants $B\geq1$ and $\eta>0$ such that
\begin{equation}\label{eq:Znt}
 Z_n(t):=
 \sum_{I\in\mathcal{A}^n}
 \exp\left(\sup_{\omega\in[I]}S_n(g+th)(\omega)\right)
 \leq Be^{-\eta n}
 \qquad(n\geq1).
\end{equation}

Fix $0<\delta<1$. 
We use the decomposition \eqref{eqKCcoverST}
from Lemma~\ref{lem:decomposition}.
Let $A$ and $\rho$ be supplied by Lemma~\ref{coveringlem}, and set
\[r_*=\min\left\{\rho,\frac{1}{2A}\right\}.\]
By \eqref{eq:BoxCoverPower}, there exists a constant $M=M(C,t,r_*)>0$ such that
\begin{equation}\label{eq:NrC}
 N_r(C)\leq Mr^{-t}
 \qquad(0<r\leq r_*).
\end{equation}
Below we estimate the contributions from $\mathcal S_\delta$
and $\mathcal T_\delta$ separately.

We first consider the words in $\mathcal S_\delta$.  For each $I\in\mathcal{S}_\delta$, choose $\omega_I\in[I]$ with
$H_{|I|}(\omega_I)=\widetilde H(I)$
and put
$r_I=\delta r_*/\widetilde H(I)$.
Then $0<r_I\leq r_*$, and Lemma~\ref{coveringlem} gives a cover of $f_I(C)$ by at most $AG_{|I|}(\omega_I)N_{r_I}(C)$ balls of radius
\[
 A r_I H_{|I|}(\omega_I)
 =Ar_*\delta\leq\frac{\delta}{2}.
\]
Using \eqref{eq:NrC}, we obtain
\[
\begin{aligned}
 N_\delta(f_I(C))
 &\leq A G_{|I|}(\omega_I)N_{r_I}(C)\\
 &\leq AMr_*^{-t}\delta^{-t}
     G_{|I|}(\omega_I)H_{|I|}(\omega_I)^t\\
 &\leq AMr_*^{-t}\delta^{-t}
 \exp\left(\sup_{\omega\in[I]}S_{|I|}(g+th)(\omega)\right).
\end{aligned}
\]
Summing over $\mathcal{S}_\delta$, then over all word lengths, and using \eqref{eq:Znt}, gives
\begin{equation}\label{eq:coverSdelta}
 N_{\delta}\left(\bigcup_{I\in\mathcal{S}_\delta}f_I(C)\right)\leq \sum_{I\in\mathcal{S}_\delta}N_\delta(f_I(C))
 \leq C_t\delta^{-t}
\end{equation}
for a constant $C_t$ independent of $\delta$.

We next consider the words in $\mathcal T_\delta$. For $n\in\N$,  define
\[
 \Delta_n(h)
 =\sup\bigl\{|S_nh(\omega)-S_nh(\tau)|: \omega,\tau\in\Sigma,
             \omega|n=\tau|n\bigr\}.
\]
Set $\Delta_0(h)=0$. 
Since $h$ is continuous, we assert that 
\begin{equation}\label{eq:Delta-nh}
 \Delta_n(h)=o(n).
\end{equation}
Indeed,  for $m\in\N$, let 
\[
 {\rm var}_m(h)=\sup\{|h(\omega)-h(\tau)|: \omega,\tau\in \Sigma, \omega|m=\tau|m\}.
\]
Then ${\rm var}_m(h)\to0$ and thus by  Ces\`aro convergence
\[
 \Delta_n(h)\leq\sum_{m=1}^{n}{\rm var}_m(h)=o(n),
\]
proving the assertion. 

Fix $I\in\mathcal{T}_\delta$ with $|I|=n$, and choose
$\omega_I\in[I]$.
For $n\geq2$, since $\omega_I\in[I^-]$, the definition of
$\Delta_{n-1}(h)$ gives
\begin{align*}
	\log H_{n-1}(\omega_I)
	&\geq
	\sup_{\tau\in[I^-]}\log H_{n-1}(\tau)-\Delta_{n-1}(h)\\
	&=\log\widetilde H(I^-)-\Delta_{n-1}(h)\\
	&\geq\log\delta-\Delta_{n-1}(h).
\end{align*}
Thus $H_{n-1}(\omega_I)
\geq \delta e^{-\Delta_{n-1}(h)}$.
For $n=1$, the same inequality follows from
$H_0=1$, $\Delta_0(h)=0$, and $\delta<1$.
Consequently,
\begin{equation}\label{eq:HnOmegaI}
	H_n(\omega_I)
	=H_{n-1}(\omega_I)H(\sigma^{n-1}\omega_I)
	\geq H_-\delta e^{-\Delta_{n-1}(h)}.
\end{equation}

Applying Lemma~\ref{coveringlem} with $D=Z$ and initial radius $r_*$ gives a cover of $f_I(Z)$ by at most
$A N_{r_*}(Z)G_n(\omega_I)$
balls of radius
\[
 Ar_*H_n(\omega_I)<Ar_*\delta\leq\frac{\delta}{2}.
\]
Moreover, by \eqref{eq:HnOmegaI},
\[
\begin{aligned}
 G_n(\omega_I)
 &=H_n(\omega_I)^{-t}
   \exp\bigl(S_n(g+th)(\omega_I)\bigr)\\
 &\leq H_-^{-t}\delta^{-t}e^{t\Delta_{n-1}(h)}
 \exp\left(\sup_{\omega\in[I]}S_n(g+th)(\omega)\right).
\end{aligned}
\]

By \eqref{eq:Delta-nh}, there exists $N_t\geq1$ such that
\[
t\Delta_{n-1}(h)\leq\frac{\eta n}{2}
\qquad(n\geq N_t).
\]
Combining this with \eqref{eq:Znt}, we obtain
\[
e^{t\Delta_{n-1}(h)}Z_n(t)
\leq B\exp\bigl(-\eta n+t\Delta_{n-1}(h)\bigr)
\leq Be^{-\eta n/2}
\qquad(n\geq N_t).
\]
It follows that
\[
\sum_{n=1}^{\infty}e^{t\Delta_{n-1}(h)}Z_n(t)<\infty.
\]
Consequently,
\begin{equation}\label{eqNTdelta}
 N_{\delta}\left(\bigcup_{I\in\mathcal{T}_\delta}f_I(Z)\right)\leq \sum_{I\in\mathcal{T}_\delta}N_\delta(f_I(Z))\leq C_t'\delta^{-t}
\end{equation}
with $C_t'$ independent of $\delta$.

Finally, $N_\delta(C)\leq M\delta^{-t}$ for sufficiently small $\delta$.  Combining this with \eqref{eqKCcoverST}, \eqref{eq:coverSdelta}, and \eqref{eqNTdelta}, we obtain
\[
 N_\delta(K_C)\leq C_t''\delta^{-t}.
\]
Thus $\overline{\dim}_{\mathrm B} K_C\leq t$.  Letting $t\downarrow\max\{q,c\}$ proves the proposition.
\end{proof}

\subsection{Passage to the singularity dimension}\label{sec:block-reduction}

In this subsection, we prove Theorem~\ref{thm:main} by first applying Proposition~\ref{prop:additive-upper} to the block systems
and then passing to the singularity dimension.

Let $s=\dim_{\mathrm S}\Phi$.
If $s\geq d$, then there is nothing to prove since $\overline{\dim}_{\mathrm B} K_C\leq d$ always holds.  Henceforth we assume
$s<d$
and put $k=\lfloor s\rfloor$.

For $n\geq1$, consider the IFS
\[\Phi^{(n)}:=\{f_I:I\in\mathcal{A}^n\},\]
where distinct word labels are retained even when
the corresponding maps coincide.
Then $\Phi^{(n)}$ is a $C^1$ IFS on $Z$ whose homogeneous attractor is still $K$.  Define
\[
 C^{[n]}
 =\bigcup_{j=0}^{n-1}\ \bigcup_{J\in\mathcal{A}^j}f_J(C).
 \]
Iterating \eqref{eq:InhomogeneousIdentity} $n$ times gives
\[K_C=C^{[n]}\cup\bigcup_{I\in\mathcal{A}^n}f_I(K_C).\]
Thus $K_C$ is also the inhomogeneous attractor of  $\Phi^{(n)}$ with condensation set $C^{[n]}$.  Since $C^{[n]}$ is a finite union of Lipschitz images of $C$ and contains $C$, we have
\[\overline{\dim}_{\mathrm B}C^{[n]}=\overline{\dim}_{\mathrm B}C=:c.\]

Let $\tau_n$ be the left shift on $\Sigma^{(n)}:=(\mathcal{A}^n)^{\N}$,
and let $\iota_n:\Sigma^{(n)}\to\Sigma$ be the natural concatenation map.
Then $\iota_n$ is a homeomorphism satisfying
\[
\iota_n\circ\tau_n=\sigma^n\circ\iota_n.
\]
Thus $(\Sigma^{(n)},\tau_n)$ is topologically conjugate to
$(\Sigma,\sigma^n)$, and for every continuous function
$\psi:\Sigma\to\R$,
\[
P(\Sigma^{(n)},\tau_n,\psi\circ\iota_n)=P(\Sigma,\sigma^n,\psi).
\]
We henceforth identify  $(\Sigma^{(n)},\tau_n)$ with $(\Sigma,\sigma^n)$ and write $P_{\sigma^n}(\psi)=P(\Sigma, \sigma^n,\psi)$. Under this identification, the coding map of $\Phi^{(n)}$ is $\Pi$.

Put $Q=(4d)^d$. Define   $G^{(n)}, H^{(n)}:\Sigma\to (0,\infty)$ by 
\begin{equation*}
G^{(n)}(\omega)=\frac{Q\phi^k(D_{\Pi(\sigma^n\omega)}f_{\omega|n})}
          {\alpha_{k+1}(D_{\Pi(\sigma^n\omega)}f_{\omega|n})^k}, \qquad H^{(n)}(\omega)=\alpha_{k+1}(D_{\Pi(\sigma^n\omega)}f_{\omega|n}).
\end{equation*}
These are the one-step covering quantities for the block IFS
$\Phi^{(n)}$; in general, they are not the products $G_n,H_n$
introduced in Section~\ref{sec:covering}. For each fixed $n\geq 1$, they are continuous and positive, with $G^{(n)}(\omega)\geq Q>1$, and 
\begin{equation}\label{eqHnOmega}
0<H^{(n)}(\omega)\leq \|D_{\Pi(\sigma^n\omega)}f_{\omega|n}\|\leq \gamma^n<1.
\end{equation}
Consequently, by the standard properties of topological pressure~\cite[Theorem 9.7]{Walters82},
 the function $q\mapsto P_{\sigma^n}\bigl(\log G^{(n)}+q\log H^{(n)}\bigr)$, $q\in [0,\infty)$, is continuous and strictly decreasing from a positive value to $-\infty$ as $q$ ranges from $0$ to $+\infty$. 
Let $q_n>0$ be its unique zero:
\begin{equation}\label{eq:qnZero}
 P_{\sigma^n}\bigl(\log G^{(n)}+q_n\log H^{(n)}\bigr)=0.
\end{equation}
Applying Proposition~\ref{prop:additive-upper} to $\Phi^{(n)}$ and $C^{[n]}$ gives
\begin{equation}\label{eq:qnbound}
\overline{\dim}_{\mathrm B}K_C\leq\max\{c,q_n\}
 \qquad (n\geq1).
\end{equation}

To deduce Theorem~\ref{thm:main} from \eqref{eq:qnbound},
it remains to show that $q_n\to s$.
The following lemma is the full-shift specialization of
\cite[Proposition~3.1]{FengSimonI}: take
$X=\Sigma$ and $\mathcal G=\mathcal G^s$ defined in \eqref{eqGs}. In this case, the set $X^{(n)}$ in that proposition equals $\Sigma$ for every $n\geq 1$. We give a direct proof for completeness.

\begin{lem}\label{lem:block-pressure}
Let $s=\dim_{\rm S}\Phi$ and let $\mathcal{G}^s=\{g_n^s\}_{n\geq 1}$ be defined in \eqref{eqGs}. 
Then
\[\lim_{n\to\infty}\frac1nP_{\sigma^n}(g_n^s)=\inf_{n\geq 1}\frac{1}{n}P_{\sigma^n}(g_n^s)
 =P_{\sigma}(\mathcal{G}^s)=0.
\]
In particular, $P_{\sigma^n}(g_n^s)\geq 0$ for every $n\geq 1$. 
\end{lem}

\begin{proof}
Recall that $g_n^s(\omega)=\log\phi^s\!\left(D_{\Pi(\sigma^n\omega)}f_{\omega|n}\right)$ and that  $\mathcal{G}^s$ is a continuous subadditive potential.
For $j\geq 1$, define
	\[
	\mathcal Z_j=\sum_{I\in\mathcal{A}^j}
	\exp\left(\sup_{\omega\in[I]}g_j^s(\omega)\right).
	\]
	For each fixed $n\geq1$, the definition of ordinary pressure on the block shift gives 
	\[
	P_{\sigma^n}(g_n^s)
	=\lim_{m\to\infty}\frac1m\log\mathcal W_{m,n},
	\]
	where 
	\[
	\mathcal W_{m,n}:=\sum_{I\in\mathcal{A}^{mn}}
	\exp\left(\sup_{\omega\in[I]}
	\sum_{r=0}^{m-1}g_n^s(\sigma^{rn}\omega)\right).
	\]
	
By subadditivity,
\[
g_{mn}^s(\omega)\leq
\sum_{r=0}^{m-1}g_n^s(\sigma^{rn}\omega),
\]
and hence $\mathcal Z_{mn}\leq\mathcal W_{m,n}$.
To bound $\mathcal W_{m,n}$ from above, write each word of length $mn$
	as $I_1\cdots I_m$, with $I_r\in\mathcal{A}^n$, $r=1,\ldots, m$. If $\omega\in[I_1\cdots I_m]$, then
	$\sigma^{(r-1)n}\omega\in[I_r]$. Therefore
	\[
	\sup_{\omega\in[I_1\cdots I_m]}
	\sum_{r=1}^m g_n^s(\sigma^{(r-1)n}\omega)
	\leq\sum_{r=1}^m\sup_{\tau\in[I_r]}g_n^s(\tau).
	\]
Exponentiating and summing
	over all $I_1,\ldots, I_m\in \mathcal{A}^n$ gives
	\[
	\mathcal W_{m,n}\leq
	\sum_{I_1,\ldots,I_m\in\mathcal{A}^n}
	\prod_{r=1}^m\exp\left(\sup_{\tau\in[I_r]}g_n^s(\tau)\right)
	=\mathcal Z_n^m.
	\]
Combining the two estimates we obtain
	\[
	\mathcal Z_{mn}\leq\mathcal W_{m,n}\leq\mathcal Z_n^m.
	\]
	Taking logarithms, dividing by $mn$, and letting $m\to\infty$ yields
	\[
	P_\sigma(\mathcal G^s)
	\leq\frac1nP_{\sigma^n}(g_n^s)
	\leq\frac1n\log\mathcal Z_n.
	\]
	
By the definition of singularity dimension, $P_\sigma(\mathcal G^s)=0$.
Thus $P_{\sigma^n}(g_n^s)\geq0$ for every $n$.
By the definition of subadditive pressure,
$\lim_{n\to\infty}(\log\mathcal Z_n)/n=P_{\sigma}(\mathcal{G}^s)=0$.
Hence
$\lim_{n\to\infty}P_{\sigma^n}(g_n^s)/n=0$.
Since this sequence is non-negative and converges to zero, its infimum is zero. This proves the lemma. 
\end{proof}

We now show that the pressure zeros $q_n$ converge to $s$.
The argument follows the proof of \cite[Theorem~1.1]{FengSimonI}, specialized to the full shift.

\begin{prop}\label{prop:qn-s}
We have $q_n>s$ for every $n\geq 1$ and $\lim_{n\to\infty}q_n=s$. 
\end{prop}

\begin{proof}
Since $s\in[k,k+1)$, the definition of the singular-value function gives
\begin{equation}\label{eqGnsHn}
\log G^{(n)}+s\log H^{(n)}=g_n^s+\log Q.
\end{equation}
By the constant-shift property of topological pressure and Lemma~\ref{lem:block-pressure},
	\[
	P_{\sigma^n}\bigl(\log G^{(n)}+s\log H^{(n)}\bigr)
	=P_{\sigma^n}(g_n^s)+\log Q>0.
	\]
Since the function $q\mapsto P_{\sigma^n}\bigl(\log G^{(n)}+q\log H^{(n)}\bigr)$ is strictly decreasing, we have $q_n>s$.
	
Set $v_n=\log G^{(n)}+q_n\log H^{(n)}$. Since $q_n>s$, 
	\eqref{eqGnsHn} and \eqref{eqHnOmega} give
	\begin{align*}
			v_n
			=g_n^s+\log Q+(q_n-s)\log H^{(n)}\leq g_n^s+\log Q-(q_n-s)n|\log\gamma|.
	\end{align*}
	By monotonicity of topological pressure and \eqref{eq:qnZero}, we obtain
	\[
	0=P_{\sigma^n}(v_n)
	\leq P_{\sigma^n}(g_n^s)+\log Q-(q_n-s)n|\log\gamma|.
	\]
Therefore,
\[0<q_n-s\leq \frac{P_{\sigma^n}(g_n^s)+\log Q}{n|\log\gamma|}.\]
The right-hand side tends to zero
	by Lemma~\ref{lem:block-pressure}, and hence $\lim_{n\to\infty}q_n=s$.
\end{proof}

\begin{proof}[Proof of Theorem~\ref{thm:main}]
If $s=\dim_{\mathrm S}\Phi\geq d$, there is nothing to prove.  If $s<d$, letting $n\to\infty$ in \eqref{eq:qnbound} and using Proposition~\ref{prop:qn-s} gives 
$\overline{\dim}_{\mathrm B}K_C\leq\max\{c,s\}$,
as desired.
\end{proof}

\section{Orbital sets driven by subshifts}\label{sec:subshift-orbits}

We retain the standing assumptions on the IFS $\Phi$.
The orbital representation \eqref{eqOrbital} suggests restricting
the compositions to words from a subshift.
Let $X\subset\Sigma$ be non-empty and compact, with
$\sigma X\subset X$; we call such an $X$ a \emph{subshift}. For $n\geq1$, put
\[
X_n=\{\omega|n:\omega\in X\},\qquad X_0=\{\vartheta\}.
\]
Forward invariance implies that every contiguous subword of a word
in $X_n$ is also allowed in $X$.
For a compact set $C\subset Z$, define
\begin{equation}\label{eq:SubshiftOrbital}
	\mathcal O_X(C)=\bigcup_{n=0}^{\infty}\ \bigcup_{I\in X_n}f_I(C),
	\qquad E_X(C)=\Pi(X)\cup\mathcal O_X(C).
\end{equation}
We refer to $\mathcal O_X(C)$ as the \emph{orbital set} driven by $X$.
In particular, $E_\Sigma(C)=K_C$ and $E_X(\emptyset)=\Pi(X)$.

The set $E_X(C)$ is compact. This follows from the same compactness argument as for the orbital representation
\eqref{eqOrbital}:
any accumulation point arising from images $f_I(C)$
along allowed words whose lengths tend to infinity
belongs to the compact set $\Pi(X)$. Moreover, if $C\neq\emptyset$, then every point of $\Pi(X)$ is a limit
of points in $\mathcal O_X(C)$, and hence
\[E_X(C)=\overline{\mathcal O_X(C)},\qquad
	\overline{\dim}_{\mathrm B}E_X(C)
	=\overline{\dim}_{\mathrm B}\mathcal O_X(C).\]
For a general $X$, the set $E_X(C)$ need not satisfy the ordinary
inhomogeneous fixed-point equation \eqref{eq:InhomogeneousIdentity},
since arbitrary concatenations of allowed words need not be allowed.

For $t\geq0$, restrict the potential $\mathcal G^t$ from
\eqref{eqGs} to $X$ and write
\begin{equation*}
		P_X(t)=P(X,\sigma,\mathcal G^t)=\lim_{n\to\infty}\frac1n
		\log\sum_{I\in X_n}
		\sup_{\omega\in[I]\cap X}\exp\bigl(g_n^t(\omega)\bigr).
\end{equation*}
As in Section~\ref{sec:preliminaries}, $P_X(t)$ is continuous
and strictly decreasing on $[0,\infty)$, with
$P_X(0)\geq0$ and
$P_X(t)\to-\infty$ as $t\to\infty$.
Following \cite[Definition~2.5]{FengSimonI}, let
$s_X$ be its unique non-negative zero and call it the \emph{singularity dimension of $X$ with respect to $\Phi$}.
Clearly, $s_\Sigma=\dim_{\mathrm S}\Phi$.

In the self-similar setting, the construction in \eqref{eq:SubshiftOrbital} and the
corresponding upper box-dimension bound were studied by
Dubey and Verma \cite{DubeyVerma}.
The following theorem establishes the analogous bound
for $C^1$ IFSs, extending Theorem~\ref{thm:main} to
orbital sets driven by subshifts.

\begin{thm}\label{thm:subshift-orbits}
	Let $X\subset\Sigma$ be a non-empty compact subshift and let
	$C\subset Z$ be compact. Then
	\begin{equation}\label{eq:subshift-upper}
		\overline{\dim}_{\mathrm B}E_X(C)
		\leq\max\{s_X,\overline{\dim}_{\mathrm B}C\}.
	\end{equation}
\end{thm}

\begin{proof}
	Put $c=\overline{\dim}_{\mathrm B}C$ and fix
	$t>\max\{s_X,c\}$. For each $m\geq1$, define
	\begin{equation*}
		X^{(m)}=\{I_1I_2\cdots:I_j\in X_m\text{ for all }j\geq1\},
		\qquad \Phi_{X,m}=(f_I)_{I\in X_m}.
	\end{equation*}
	Maps in the IFS $\Phi_{X,m}$ are indexed by the words in $X_m$, with
	distinct labels retained even when the corresponding maps coincide.
	The concatenation map identifies the full shift $(X_m)^{\N}$ with
	$(X^{(m)},\sigma^m)$. Under this identification, the coding map of $\Phi_{X,m}$ is 
	$\Pi|_{X^{(m)}}$, and  its homogeneous attractor is $\Pi(X^{(m)})$. The system $\Phi_{X,m}$ is a $C^1$ IFS on $Z$ with common contraction constant at most $\gamma^m$.
	
	Set
	\begin{equation*}
		D_m=\bigcup_{k=0}^{m-1}\ \bigcup_{J\in X_k}f_J(C),
	\end{equation*}
	and let $E_m$ be the ordinary inhomogeneous attractor of
	$\Phi_{X,m}$ with condensation set $D_m$, so that
	\[
	E_m=D_m\cup\bigcup_{I\in X_m}f_I(E_m).
	\]
Applying the orbital representation \eqref{eqOrbital} to
$\Phi_{X,m}$ with condensation set $D_m$, we obtain
\[
E_m=\Pi(X^{(m)})\cup D_m\cup\bigcup_{p=1}^{\infty}
	\ \bigcup_{I_1,\ldots,I_p\in X_m}
	\bigl(f_{I_1}\circ\cdots\circ f_{I_p}\bigr)(D_m).
\]

We assert that
	\begin{equation}\label{eq:subshift-enclosure}
		\overline{\dim}_{\mathrm B}D_m=c,\qquad E_X(C)\subset E_m.
	\end{equation}
	Indeed, the dimension identity holds since $D_m$ is a finite union of Lipschitz images of $C$ and contains $C$.
	To verify that $E_X(C)\subset E_m$, first note that $X\subset X^{(m)}$, whence
	$\Pi(X)\subset\Pi(X^{(m)})\subset E_m$.
	For $n\geq 1$ and $W\in X_n$, write $n=pm+r$, where $0\leq r<m$, and split
	$W=I_1\cdots I_pJ$ with  $|J|=r$ and $|I_i|=m$ for $1\leq i\leq p$.
Since $\sigma X\subset X$, we have  $I_i\in X_m$ and $J\in X_r$, and hence
	\[
	f_W(C)=f_{I_1}\circ\cdots\circ f_{I_p}\bigl(f_J(C)\bigr)
	\subset f_{I_1}\circ\cdots\circ f_{I_p}(D_m)\subset E_m,
	\]
where the empty composition is interpreted as the identity
when $p=0$.
Since $C\subset D_m\subset E_m$, the inclusion follows.
	
We now choose the block length using pressure approximation.
The potential $\mathcal G^t$ is continuous and subadditive on the
	whole space $\Sigma$, so \cite[Proposition~3.1]{FengSimonI} gives
	\begin{equation*}
		\lim_{m\to\infty}\frac1m
		P(X^{(m)},\sigma^m,g_m^t)=P_X(t)<0.
	\end{equation*}
	Here $P(X^{(m)},\sigma^m,g_m^t)$ is the ordinary additive pressure
	on the full shift over the block alphabet $X_m$.
	Choose $m$ sufficiently large that this pressure is negative.
	
	Under the concatenation identification, the singular-value
	potential of $\Phi_{X,m}$ is $(g_{mj}^s)_{j\geq1}$, for $s\geq0$.
	Thus its singularity dimension $s_m=\dim_{\mathrm S}\Phi_{X,m}$
	is the zero of the subadditive pressure
	\[
	\mathscr P_m(s)
	=P\bigl(X^{(m)},\sigma^m,(g_{mj}^s)_{j\geq1}\bigr).
	\]
	Subadditivity gives
	\[
	g_{mj}^t(\omega)
	\leq\sum_{a=0}^{j-1}g_m^t(\sigma^{am}\omega),
	\qquad \omega\in X^{(m)},\quad j\geq1.
	\]
By monotonicity of pressure,
	\begin{equation*}
		\mathscr P_m(t)\leq P(X^{(m)},\sigma^m,g_m^t)<0.
	\end{equation*}
	It follows that $s_m<t$. Applying Theorem~\ref{thm:main} to
	$\Phi_{X,m}$ and $D_m$, and using \eqref{eq:subshift-enclosure},
	we obtain
	\[
	\overline{\dim}_{\mathrm B}E_X(C)
	\leq\overline{\dim}_{\mathrm B}E_m
	\leq\max\{s_m,c\}<t.
	\]
	Letting $t\downarrow\max\{s_X,c\}$ proves
	\eqref{eq:subshift-upper}. 
\end{proof}

\section{Consequences}\label{sec:consequences}

In this section, we give conditions under which the upper bound in Theorem~\ref{thm:main} is attained.

\begin{cor}\label{cor1} The following properties hold. 
\begin{itemize}
	\item[{\rm (i)}] If $\overline{\dim}_{\mathrm B} C\geq\dim_{\mathrm S}\Phi$,
	then $\overline{\dim}_{\mathrm B}K_C=\overline{\dim}_{\mathrm B}C$.
	\item[{\rm (ii)}] If $\overline{\dim}_{\mathrm B} K=\dim_{\mathrm S}\Phi$,
	then
	\[\overline{\dim}_{\mathrm B}K_C
	=\max\{\dim_{\mathrm S}\Phi,\overline{\dim}_{\mathrm B}C\}
	=\max\{\overline{\dim}_{\mathrm B}K,\overline{\dim}_{\mathrm B}C\}.\]
\end{itemize}
\end{cor}

\begin{proof}
Theorem~\ref{thm:main} gives the required upper bounds
in both cases. The reverse inequalities follow from
$C\subset K_C$ in \textup{(i)} and
$K\cup C\subset K_C$ in \textup{(ii)}.
\end{proof}

Recent work of Li, Pan, Tong, and Xu~\cite{LiPanTongXu} gives broad classes of planar nonlinear systems for which the attractor attains the singularity dimension.  Combining their result with Corollary~\ref{cor1}(ii) gives the following.

\begin{cor}\label{cor2}
Let $\Phi=\{f_i\}_{i=1}^{\ell}$ be a planar $C^r$ IFS, $r>1$, satisfying
the standing assumptions of this paper, as well as
the strong separation condition, uniform non-conformality,
and weak irreducibility in the sense of \cite[Section~2.8]{LiPanTongXu}.
Let $s_0=\dim_{\rm S}\Phi$.  Then, for every compact condensation set $C\subset Z$,
\begin{equation}\label{eq:planarKC}
\overline{\dim}_{\mathrm B}K_C=\max\{s_0,\overline{\dim}_{\mathrm B}C\}.
\end{equation}
If $\dim_{\mathrm B}C$ exists, then $\dim_{\mathrm B} K_C$ also exists and
\begin{equation}\label{eq:planarBKC}
\dim_{\mathrm B} K_C=\max\{s_0,\dim_{\mathrm B} C\}.
\end{equation}
\end{cor}

\begin{proof}
For every $n\geq 1$,  $I\in\mathcal A^n$, and $s\in [0,2]$, we have
\[
\sup_{\omega\in[I]}\exp(g_n^s(\omega))
=\sup_{x\in K}\phi^s(D_xf_I).
\]
Thus the pressure in \cite[(2.17)]{LiPanTongXu} agrees with
$P_\sigma(\mathcal G^s)$ on $[0,2]$,
and its unique zero equals $s_0$.
By \cite[Theorem D]{LiPanTongXu},
\[
\dim_{\mathrm H}K=\dim_{\mathrm B} K=s_0.
\]
Then \eqref{eq:planarKC} follows from Corollary~\ref{cor1}(ii).  If $\dim_{\mathrm B} C$ exists, then monotonicity gives
\[
\underline{\dim}_{\mathrm B} K_C\geq\max\{\dim_{\mathrm B} K,\dim_{\mathrm B} C\}=\max\{s_0,\dim_{\mathrm B}C\}.
\]
Together with \eqref{eq:planarKC}, this proves \eqref{eq:planarBKC}.
\end{proof}

{\noindent \bf  Acknowledgements}. The author was supported by the National Natural Science
Foundation of China (Grant No.~12301110). AI tools were used during the preparation of the paper for mathematical checking,  editorial organization, and typesetting. The author takes full responsibility for the content of the paper.

\end{document}